\documentclass[10pt]{article}
\usepackage{geometry}                % See geometry.pdf to learn the layout options. There are lots.
\usepackage{graphicx}
\usepackage{amssymb,amsthm,amsmath}
\usepackage{epstopdf}
\usepackage{pdfsync}

\usepackage{color}

\def\sau{{\mathop{u}\limits^{\scriptscriptstyle{\,\circ}}}}

\title{Homogenization of random convolution energies\\
in heterogeneous and perforated domains}
\author{
{\sc Andrea Braides
}
\\ Dipartimento di Matematica,
 Universit\`a di Roma `Tor Vergata'\\
via della Ricerca Scientifica, 00133 Rome, Italy\\
\and
{\sc Andrey Piatnitski}
\\
The Arctic University of Norway, UiT,  Campus
Narvik,\\ P.O. Box 385, Narvik 8505, Norway \\ and \\ Institute for Information Transmission Problems
of RAS,\\ 127051 Moscow, Russia %\ \ \ \ email: \emph{apiatnitski@gmail.com}
\\
}
\date{%\today
}                                      % Activate to display a given date or no date

\newtheorem{definition}{Definition}[section]
\newtheorem{lemma}[definition]{Lemma}
\newtheorem{theorem}[definition]{Theorem}
\newtheorem{proposition}[definition]{Proposition}
\newtheorem{corollary}[definition]{Corollary}
\newtheorem{remark}[definition]{Remark}

\def\Domain{D}

\def\e{\varepsilon}
\def\dx{\,dx}

\def\ZZ{\mathbb{Z}}
\def\NN{\mathbb{N}}
\def\rr{\mathbb{R}}

\def \trait (#1) (#2) (#3){\vrule width #1pt height #2pt depth #3pt}
\def \qed{\hfill
        \trait (0.1) (6) (0)
        \trait (6) (0.1) (0)
        \kern-6pt
        \trait (6) (6) (-5.9)
        \trait (0.1) (6) (0)
\medskip}

\begin{document}
\maketitle

\noindent
{\bf Abstract.} We prove a homogenization theorem for a class of quadratic convolution energies with random coefficients.
Under suitably stated hypotheses of ergodicity and stationarity we prove that the $\Gamma$-limit of such energy is almost surely a deterministic quadratic Dirichlet-type integral functional, whose integrand can be characterized through an asymptotic formula.
The proof of this characterization relies on results on the asymptotic behaviour of  
subadditive processes.
The proof of the limit theorem uses a blow-up technique common for local energies, that can be extended to this `asymptotically-local' case.
As a particular application we derive a homogenization theorem on random perforated domains.
\bigskip


\section{Introduction} In this paper we consider random energies of convolution type. Such energies may be interpreted for example in the context of mathematical models in population dynamics where macroscopic properties can be reduced to studying the evolution of the first-correlation functions describing the population density $u$ in the system \cite{KKP2008, FKK2012}.
Our model energies are defined on $L^2$-functions in a reference domain $D$ and are of the form
\begin{eqnarray}\label{def-Fe-solid-random1}
{1\over \e^{d+2}}\int_{D\times D}B^\omega \Bigl({x\over\e},{y\over\e}\Bigr)\, a\Bigl({y-x\over\e}\Bigr)
(u(y)-u(x))^2dy \dx,
\end{eqnarray}
or
\begin{eqnarray}\label{def-Fe-solid-random2}
{1\over \e^{d+2}}\int_{(D\cap \e E^\omega)\times (D\cap \e E^\omega)}a\Bigl({y-x\over\e}\Bigr)
(u(y)-u(x))^2dy \dx.
\end{eqnarray}
%where $E^\omega$ is a random perforated domain consisting of a unique connected component.
Here  $a:\rr^d\to \rr$  is a {\em convolution kernel}  which describes the strength of the interaction at a given distance and  $\e$ is a scaling parameter. In order that the limit of energies above be well-defined on $H^1(\Omega)$ we require that
\begin{equation}\label{finitesm}
\int_{\rr^d}a(\xi)(1+|\xi|^2)\,dx<+\infty.
\end{equation}
In \eqref{def-Fe-solid-random1} the strictly positive coefficient $B^\omega$
represents the features of the environment, while in \eqref{def-Fe-solid-random2} $E^\omega$ is a random perforated domain
giving the regions where interaction actually occurs, both depending on the realization of a random variable.
Note that functionals \eqref{def-Fe-solid-random2} can be also written as \eqref{def-Fe-solid-random1} with the degenerate coefficient $B^\omega (x,y)= \chi_{E^\omega}(x)\chi_{E^\omega}(y)$, where $\chi_E$ denotes the characteristic function of $E$.
Note that more in general we may consider oscillations on a different scale than $\e$; e.g.~taking coefficients $B^\omega ({x/\delta},{y/\delta})$ with $\delta=\delta_\e$, but the case when these two scales differ can be treated more easily by a separation-of-scale argument.

The effect of the scaling parameter $\e$ as $\e\to 0$ is twofold, on one hand producing a local limit model as the convolution kernel concentrates, and on the other hand ensuring a homogenization effect through the oscillations provided by $B^\omega$.
To illustrate the first issue, we may consider the underlying energies (those with the perturbation $B^\omega$ set to $1$)
\begin{eqnarray}\label{def-Fe-solid}
{1\over \e^{d+2}}\int_{D\times D} a\Bigl({y-x\over\e}\Bigr)
(u(y)-u(x))^2dy \dx.
\end{eqnarray}
 We note that if $u\in C^1(D)$ then
$u(y)-u(x)\approx\langle\nabla u(x),y-x\rangle$ and, using the change of variables $y=x+\e\xi$,
\begin{eqnarray}\label{def-Fe-solid-2}
\lim_{\e\to0}{1\over \e^{d+2}}\int_{D\times D} a\Bigl({y-x\over\e}\Bigr)
(\langle\nabla u(x),y-x\rangle)^2dy \dx= \int_{D} \int_{\rr^d}a(\xi)
(\langle\nabla u(x),\xi\rangle)^2d\xi \dx,
\end{eqnarray}
so that the quadratic functional
 \begin{eqnarray}\label{def-Fe-solid-3}
 \int_D \langle A\nabla u,\nabla u\rangle \dx, \qquad\hbox{ with }\qquad
 \langle A z, z\rangle= \int_{\rr^d}a(\xi)
(\langle z,\xi\rangle)^2d\xi,
\end{eqnarray}
gives an approximation of \eqref{def-Fe-solid}. Conversely, we may think of \eqref{def-Fe-solid} as giving a more general form of quadratic energies allowing for interactions between points at scale $\e$. In terms of $\Gamma$-convergence this computation can be extended to a $\Gamma$-limit result and obtain the corresponding convergence of minimum problems. To that end we will suppose that  $a:\rr^d\to \rr$  satisfies
\begin{equation}\label{ass-aaa}
0 \le a(\xi)\le C {1\over (1+|\xi|)^{d+2+\kappa}}.
\end{equation}
for some $C,\kappa>0$ (which is a quantified version of \eqref{finitesm}), and
\begin{eqnarray}\label{ass-aa}
 a(\xi)\ge c>0\quad \hbox{if } |\xi|\le r_0
\end{eqnarray}
for some $r_0>0$ and $c>0$.

In a $\Gamma$-convergence context energies \eqref{def-Fe-solid} have been considered as an approximation of a Dirichlet-type integral in phase-transition problems (see e.g.~\cite{AlbBel}) and more recently in connection with minimal-cut problems in Data Science \cite{GS}.
Limits of energies  similar to \eqref{def-Fe-solid}, of the form
\begin{eqnarray}\label{def-Fe-solid-Brezis}
{1\over \e^{d}}\int_{D\times D} a\Bigl({y-x\over\e}\Bigr)
\Bigl|{u(y)-u(x)\over y-x}\Bigr|^2dy \dx,
\end{eqnarray}
have also been studied by  Bourgain et al.~\cite{BBM} as an alternative definition of the $L^p$-norm of the gradient of a Sobolev function, while, in the context of Free-Discontinuity Problems, functionals of the form
\begin{eqnarray}\label{def-Fe-solid-Gobbino}
{1\over \e^{d}}\int_{D\times D} a\Bigl({y-x\over\e}\Bigr)
\min\Bigl\{\Bigl|{u(y)-u(x)\over y-x}\Bigr|^2, {1\over\e}\Bigr\}dy \dx,
\end{eqnarray}
have been proved to provide an approximation of the Mumford-Shah functionals by Gobbino \cite{Go} after a conjecture by De Giorgi. Furthermore, discrete counterparts of functionals \eqref{def-Fe-solid}; i.e., energies of the form
\begin{eqnarray}\label{def-Fe-lattice}
{1\over \e^{d+2}}\sum_{i,j\in \e\cal L} a_{ij}(u_i-u_j)^2
\end{eqnarray}
where $\cal L$ is a $d$-dimensional lattice have been widely investigated (see e.g.~\cite{PiatnitskiRemy,AliCic,BraidesSeoul,BK}) as a discrete approximation of quadratic integral functionals. Such type of functionals or the corresponding operators have been analyzed in different ways under various inhomogeneity and randomness assumptions (see e.g.~\cite{Kozlov,PiatnitskiRemy,BBL,BF,ACG,BBL,GS,BCR}).

In our case, we will prove a general homogenization result, which, under proper stationarity and ergodicity assumptions, will comprise both random coefficients and random perforated domains as in \eqref{def-Fe-solid-random1}, assuming that $B^\omega$ satisfies $0<\lambda_1\le B^\omega (x,y) \le \lambda_2<+\infty$, and \eqref{def-Fe-solid-random2} where $E^\omega$ is a {\em random perforated domain} consisting of a unique connected component. The limit behaviour of these energies is described by their $\Gamma$-limit in the $L^2(D)$ topology as a standard elliptic integral, of the form
\begin{equation}
F_{\rm hom}(u)=\int_\Domain \langle A_{\rm hom}\nabla u,\nabla u\rangle\dx.
\end{equation}
The matrix $ A_{\rm hom}$ is characterized by an asymptotic formula obtained using a limit theorem for subadditive processes.
The choice of the $L^2(D)$ topology is justified by the coerciveness of the convolution energies, which ensures the convergence of minimum problems.

\smallskip
The plan of the paper is as follows. In Section \ref{setting_sec} we define the general form of the random functionals that we are going to consider. Section \ref{s_compa} is devoted to the statement and proof of a compactness theorem.
The proof of this result follows closely that of the compactness result for non-linear convolution energies used to approximate Free-Discontinuity Problems obtained by Gobbino \cite{Go,LN98}; thanks to the quadratic growth conditions on the energies we can improve that result from $L^1$ to $L^2$ compactness. In Section \ref{s_poinc} we prove Poincar\'e and Poincar\'e-Wirtinger inequalities, which, together with the compactness result, justify the application of the direct method of the Calculus of Variations to minimum problems, and hence the asymptotic study of convolution energies in terms of $\Gamma$-convergence. In Section \ref{s_hom_formula} we use the stationarity and ergodicity properties of the energies to prove the existence of an asymptotic homogenization formula giving a deterministic homogeneous integrand using results on the asymptotic behaviour of almost-subadditive processes in \cite{KrPy}. The formula is used in Section \ref{hom} to prove the homogenization theorem using an adaptation to (non-local) homogenization problems of the blow-up technique of Fonseca and M\"uller \cite{FM,BMS}. Finally, in Section \ref{hom_perf} we remark that the result can be applied to the homogenization of random perforated domains.

\section{Setting of the problem}\label{setting_sec}
Let $\Domain$ be an open subset of $\rr^d$. For all $\e>0$ and $u\in L^2(\Domain)$ we will consider convolution-type energies of the form
\begin{eqnarray}\label{def-Fe-0}
F^\omega_\e(u)={1\over \e^{d+2}}\int_{\Domain}\int_{\Domain} b^\omega\Bigl({x\over\e},{y\over\e}\Bigr)
(u(y)-u(x))^2\,dy \dx,
\end{eqnarray}
where $b^\omega$ are stationary ergodic integrands satisfying
\begin{eqnarray}\label{ass-a}
0 \le b^\omega(x,y)\le C {1\over (1+|x-y|)^{d+2+\kappa}}.
\end{eqnarray}
More precisely, given a probability space $(\Omega,\mathcal{F},{\bf P})$ with an ergodic dynamical system $\tau_x$,
we assume that
\begin{equation}\label{def_b}
b^\omega(x,y)={\tt b}(\tau_x\omega, \tau_y\omega, x-y),
\end{equation}
where ${\tt b}(\omega_1, \omega_2, \xi)$ is a function define on $\Omega\times\Omega\times\mathbb R^d$
such that
\begin{eqnarray}\label{ass-a_bis}
0 \le {\tt b}(\omega_1,\omega_2,\xi)\le C {1\over (1+|\xi|)^{d+2+\kappa}}.
\end{eqnarray}
In order to make the definition of a function $b$ in  \eqref{def_b} well defined we need additional assumptions  on ${\tt b}$.
One option is to assume that ${\tt b}(\omega_1,\omega_2,\xi)={\tt b}_1(\omega_1) {\tt b}_2(\omega_2) a(\xi)$,
where ${\tt b}_1$ and ${\tt b}_2$ are nonnegative bounded random variables, and $a(\xi)$ is a measurable function in
$\mathbb R^d$ that satisfies  estimate \eqref{ass-aaa}.
Another option is to assume that $\Omega$ is a topological space, the group $\tau_x\omega$ is continuous in $x$,
and the function   ${\tt b}={\tt b}(\omega_1,\omega_2,\xi)$ is continuous in $\omega_1$ and $\omega_2$
and measurable in $\xi$ and ${\tt b}(\omega_1,\omega_2,\xi)\le a(\xi)$ with $a$ as above.
In both cases the definition of $b^\omega$ in \eqref{def_b} makes sense.

In order to obtain coerciveness properties which allow to include in our results both types of models
\eqref{def-Fe-solid-random1} and \eqref{def-Fe-solid-random2}; i.e., with integrands

$\bullet$ $b^\omega(x,y)= B^\omega (x,y) a(x-y)$
with $0<\lambda_1\le B^\omega (x,y) \le \lambda_2<+\infty$, or

$\bullet$ $b^\omega(x,y)= \chi\big._{E^\omega} (x)\chi\big._{E^\omega} (y) a(x-y),$

\noindent
we will make the following abstract assumption.

\begin{definition}\label{coerciveness}
 We say that $b^\omega$ is a {\em coercive energy function} if there exist constants $C$ and $\Xi_0$ such that for all $U$ open subsets of $\rr^d$, $z\in \rr^d$, $\Xi\geq \Xi_0$ and
$u\in L^2(U)$ satisfying the boundary condition
$$
u(x)= \langle z,x\rangle \hbox{\rm  \quad if \   {\rm dist}}(x,\partial U)<\Xi
$$
there exists a function $v\in L^2(U)$ satisfying the boundary condition
$$
v(x)= \langle z,x\rangle \hbox{\rm  \quad if \   {\rm dist}}(x,\partial U)<\Xi/2
$$
such that
\begin{eqnarray}\label{def-Fe-zero}
\int_{U\times U} b^\omega(x,y) (v(y)-v(x))^2\,dy \dx\le
\int_{U\times U} b^\omega(x,y) (u(y)-u(x))^2\,dy \dx,
\end{eqnarray}
and
\begin{eqnarray}\label{def-Fe-1}
\int_{\{x,y\in U:|x-y|<1\}} (v(y)-v(x))^2\,dy \dx\le C
\int_{U\times U} b^\omega(x,y) (v(y)-v(x))^2\,dy \dx.
\end{eqnarray}
\end{definition}

\begin{remark}\rm
Note that if $b^\omega(x,y)\ge C>0$ when $|x-y|<1$ and we take $u=v$ in the definition above, or if $b^\omega(x,y)=\chi_E(x)\chi_E(y)a(x-y)$ with $E$ a deterministic periodic perforated domain with $v$ a suitable extension of $u$ in the perforation constructed in \cite{2018BP} then $b^\omega$ is coercive.
\end{remark}

\begin{remark}[coerciveness]\label{co-co}\rm The terminology in Definition \ref{coerciveness} is justified by the Compactness Theorem in Section \ref{s_compa}, which ensures that if $b^\omega$ is a coercive energy function, then sequences bounded in $L^2(D)$ and for which the energy on the left-hand side of  \eqref{def-Fe-1} is equibounded admit $L^2_{\rm loc}(D)$ converging subsequences and their limit is in $H^1(D)$. \end{remark}

%We suppose that

%\def\Kz{K_0}
%$\bullet$  $\Kz$ is a compact subset of $\rr^d$ with Lipschitz boundary
%such that $(\Kz+j)\cap(\Kz+j')=\emptyset$ if $j,j'\in\ZZ^d$ and $j\neq j'$.

%We then define the {\em perforated domain}
%\begin{equation}
%E=\rr^d\setminus (\Kz+\ZZ^d).
%\end{equation}

%Equivalently, we will write
%\begin{eqnarray}\label{def-Fe}\nonumber
%F_\e(u)&=&{1\over \e^{d}}\int_{\Domain\cap\e E}\int_{\Domain\cap\e E} a\Bigl({y-x\over\e}\Bigr)
%\Bigl({u(y)-u(x)\over\e}\Bigr)^2dy \dx\\ \nonumber
%&=&{1\over \e^{d+2}}\int_{\Domain\cap\e E}\int_{\Domain\cap\e E-x} a\Bigl({\xi\over\e}\Bigr)
%(u(x+\xi)-u(x))^2d\xi \dx
%\\
%&=&\int_{\Domain\cap\e E}\int_{{1\over\e}(\Domain\cap\e E-x)} a(\xi)
%\Bigl({u(x+\e\xi)-u(x)\over\e}\Bigr)^2d\xi \dx,
%\end{eqnarray}
%this last version being useful when integrability properties of $a$ are used.
%
%\bigskip
%In the following we will prove extension, compactness and convergence properties for the functionals $F_\e$ under the following hypothesis:
%\begin{equation}\label{ass-aa}
%\hbox{there exist a constant $c>0$ and $r_0>0$ such that
% }a(z)\ge c \hbox{ if }|x|\le r_0.
%\end{equation}

%If such an hypothesis is removed, indeed the functionals $F_\e$ may be `degenerate', as shown in Example \ref{}.

\subsection{Notation}Unless otherwise stated $C$ denotes a generic strictly positive constant independent of the parameters of the problem taken into account.

$Q_T=[-T/2,T/2]^d$ denotes the $d$-dimensional coordinate cube centered in $0$ and with side-length~$T$. If $T=1$ then we write $Q=Q_1$.

If $x,y\in \rr^d$ then $|y-x|_1=\sum_{j=1}^d|y_j-x_j|$.

$\lfloor t\rfloor$ denotes the integer part of $t\in\rr$.

$\chi_A$ denotes the characteristic function of the set $A$.

For all $t>0$ and $\Domain$ open subset of $\rr^d$ we denote
$\Domain(t)=\{x\in\Domain: {\rm dist}(x,\partial \Domain)>t\}$.

As a shorthand, the notation $\{P(\xi)\}$ will stand for $\{\xi\in\rr^d: P(\xi) \hbox{ holds}\}$ if no confusion may arise.

\section{A compactness theorem}\label{s_compa}
Let $\Domain$ be an open set with Lipschitz boundary. We show that families of functions  that have bounded energies of the type \eqref{def-Fe-solid} is compact in $L^2_{\rm loc}(\Domain)$.
To this end, for $0<r\le\sigma$, we define the functional
$$
F_\e^{\sigma,r}(w)=\int_{\Domain(\sigma)}\int_{\{|\xi|\leq r\}}
\Bigl({w(x+\e\xi)-w(x)\over\e}\Bigr)^2d\xi \dx, \quad w\in L^2(\Domain).
$$
In the case when $\Domain=\rr^d$ the $L^1_{\rm loc}$-compactness can be directly obtained by comparison with finite-difference energies approximating the Mumford-Shah functional studied by Gobbino \cite{Go}. Here we follow his proof, to deduce the $L^2_{\rm loc}$-compactness.

\begin{theorem}[compactness theorem]\label{t_comp}
Let $\Domain$ be an open set with Lipschitz boundary, and assume that for a family $\{w_\e\}_{\e>0}$, $w_\e\in L^2(\Domain)$,
the estimate  \begin{equation}\label{energ_bou}
F_\e^{k\e,r}(w_\e):=\int_{\Domain(k\e)}\int_{\{|\xi|\leq r\}}
%_{\begin{array}{c}\\[-5.5mm]\scriptstyle|\xi|\le r,\\[-0.9mm]\scriptstyle x+\e\xi\in\Domain\end{array}}
\Bigl({w_\e(x+\e\xi)-w_\e(x)\over\e}\Bigr)^2d\xi \dx\le C
\end{equation}
is satisfied with some  $k>0$ and $r>0$.
Assume moreover that the family $\{w_\e\}$ is bounded in $L^2(\Domain)$.
Then for any sequence ${\e_j}$ such that $\e_j>0$ and $\e_j\to0$, as $j\to\infty$,  and for any open subset
$\Domain'\Subset\Domain$  the set $\{w_{\e_j}\}_{j\in\mathbb N}$ is relatively compact in $L^2(\Domain')$  and every its limit point is in $H^1(\Domain)$.
\end{theorem}

Before proving the theorem we prove some auxiliary results.
We first introduce the local average of a function $u\in  L^2(\Domain)$ by
$$
\sau_\delta
%{\mathop{u}\limits^{\scriptscriptstyle{\,\circ}}}_\delta
=\int_{\{|\xi|\leq 1\}}u(x+\delta\xi)\phi(\xi)\,d\xi,
$$
where  $\phi$ is a symmetric non-negative $C_0^\infty$ function in $\mathbb R^d$ supported in the unit
ball centered at the origin, $\int\phi(\xi)\,d\xi=1$.  %We assume that function
%$u$ is equal to zero in $\mathbb\setminus\Domain$.
In our framework the function $\sau_\delta$ is well defined in $\Domain(\delta)$.
The properties of the local average operator are described in the following statement.
\begin{proposition}\label{p_stek} Let $\delta$ and $\sigma$ be positive numbers with $\delta<\sigma$.
%$\mathrm{dist}(\partial\Domain', \partial\Domain)$
Then we have
\begin{equation}\label{bliz1}
\|\sau_\delta-u\|^2_{L^2(\Domain(\sigma))}\leq C_\phi\delta^2 F^{\sigma,1}_\delta(u).
\end{equation}
  For any $\delta>0$ such that $\Domain'\subset\Domain(\delta)$  the function $\sau_\delta$ is smooth in $\Domain'$
and satisfies the inequalities
\begin{equation}\label{est_mean}
\|\sau_\delta\|_{L^\infty(\Domain')}\leq C_\phi\delta^{-\frac d2}\|u\|_{L^2(\Domain)},\qquad
\|\nabla\sau_\delta\|_{L^\infty(\Domain')}\leq C_\phi\delta^{-\frac d2 -1}\|u\|_{L^2(\Domain)}.
\end{equation}
\end{proposition}

\begin{proof}
For any $u\in L^2(\Domain)$ by the Cauchy-Schwartz inequality we have
\begin{eqnarray*}
\|\sau_\delta-u\|^2_{L^2(\Domain(\sigma))}&=&\int_{\Domain(\sigma)}\int_{\{|\xi|\leq 1\}}\int_{\{|\eta|\leq 1\}}
\big(u(x+\delta\xi)-u(x)\big)\,
\big(u(x+\delta\eta)-u(x)\big)\,\phi(\xi)\,\phi(\eta)\,d\eta\, d\xi\, dx
\\
&\leq&\delta^2\bigg(\int_{\Domain(\sigma)}\int_{\{|\xi|\leq 1\}}\int_{\{|\eta|\leq 1\}}
\Big(\frac{u(x+\delta\xi)-u(x)}{\delta}\Big)^2(\phi(\xi))^2dxd\xi d\eta\bigg)^\frac12\times
\\&&
\qquad\qquad\times
\bigg(\int_{\{|\xi|\leq 1\}}\int_{\{|\eta|\leq 1\}}
\Big(\frac{u(x+\delta\eta)-u(x)}{\delta}\Big)^2(\phi(\eta))^2dxd\xi d\eta\bigg)^\frac12
\\
&\le& C_\phi\delta^2 F^{\sigma,1}_\delta(u).
\end{eqnarray*}
The estimates in \eqref{est_mean} are standard.
\end{proof}
\begin{proposition}\label{p_multi}
For any $j\in\mathbb N$ such that $j\e\leq\mathrm{dist}(\Domain',\partial\Domain)-k\e$ the following inequality holds:
 \begin{equation}\label{energ_bou2}
F_{j\e}^{(j+k)\e,1}(u)\leq F_\e^{k\e,1}(u)
\end{equation}
for all $u\in L^2(\Domain)$.
\end{proposition}
\begin{proof}
Representing $u(x+j\e\xi)-u(x)$ as $(u(x+j\e\xi)-u(x+(j-1)\e\xi))+(u(x+(j-1)\e\xi)-u(x+(j-2)\e\xi))+
\ldots+(u(x+\e\xi)-u(x))$ we obtain
$$
F_{j\e}^{(j+k)\e,1}(u)\leq j\int_{\Domain((j+k)\e)}\int_{\{|\xi|\leq1\}}\sum_{m=1}^{j}
\frac{\big(u(x+m\e\xi)-u(x+(m-1)\e\xi)\big)^2}{(j\e)^2}dxd\xi
$$
$$
\leq j^2\int_{\Domain(k\e)}\int_{\{|\xi|\leq1\}}
\frac{\big(u(x+\e\xi)-u(x)\big)^2}{(j\e)^2}dxd\xi = F_\e^{k\e,1}(u)
$$
as desired.\end{proof}

\begin{proof}[Proof of Theorem {\rm\ref{t_comp}}]
One may assume without loss of generality that $r=1$.  In order to prove the compactness result, it suffices to show that, fixed $D'$, for each $\delta>0$
 there  exists a relatively compact set  $\mathcal{K}_\delta$ in $L^2(\Domain')$ such that for any
 $j\in \mathbb N$  we have
 \begin{equation}\label{tech_est}
 \|w_{\e_j}-h_j\|_{L^2(\Domain')}\leq\delta
 \end{equation}
 for some $h_j\in\mathcal{K}_\delta$.

We define   $\mathcal{K}_\delta$ as follows.
If $\e_j\geq\delta$, we set $h_j=w_{\e_j}$; otherwise,
$$
h_j={\mathop{w}\limits^{\scriptscriptstyle{\,\circ}}}_{\e_j,\delta_j} =
\int_{\{|\xi|\leq 1\}}w_{\e_j}(x+\delta_j\xi)\phi(\xi)\,d\xi,
$$
where $\delta_j=\big\lfloor\frac{\delta}{\e_j}\big\rfloor\,\e_j$.
Note that $\frac12\delta<\delta_j\leq\delta$ for any $j$ such that $\e_j<\delta$.  We finally set
$\mathcal{K}_\delta=\mathop{\bigcup}\limits_{j=1}^\infty \{h_j\}$.

It is convenient to represent $\mathcal{K}_\delta$ as a union
$\mathcal{K}_\delta=\mathcal{K}_{\delta,1}\cup\mathcal{K}_{\delta,2}$ with
$$
\mathcal{K}_{\delta,1}=\bigcup_{\{j\,:\,\e_j\geq\delta\}}h_j,\qquad
\mathcal{K}_{\delta,2}=\bigcup_{\{j\,:\,\e_j<\delta\}}h_j
$$
Since $\e_j$ tends to zero as $j\to\infty$, the first set consists of a finite number of elements
and thus is compact. By \eqref{est_mean} for any $h_j\in \mathcal{K}_{\delta,2}$ we obtain
$$
|h_j(x)|\leq C(\delta), \quad  |\nabla h_j(x)|\leq C(\delta)\quad \hbox{for all }x\in\Domain'.
$$
 Therefore, by the Arzel\`a-Ascoli theorem, the set $\mathcal{K}_{\delta,2}$ is relatively compact
 in $C(\Domain')$. Consequently, this set is also relatively compact in $L^2(\Domain')$.   This yields the desired
 relative compactness of $\mathcal{K}_{\delta}$.

 If $\e_j\geq\delta$ then $h_j=w_{\e_j}$,  and \eqref{tech_est} holds. If
 $\e_j<\delta$, then by \eqref{bliz1} we get
 $$
 \|w_{\e_j}-h_j\|\leq C_\phi\delta_j F^{(\delta_j+k\e_j)}_{\delta_j,1}(w_{\e_j}).
 $$
 Combining this inequality with  \eqref{energ_bou2} and recalling that $\delta_j=\big\lfloor\frac{\delta}{\e_j}\big\rfloor\,\e_j$,
 we obtain
 $$
 \|w_{\e_j}-h_j\|\leq C_\phi\delta_j F^{k\e_j}_{\e_j,1}(w_{\e_j})\leq C\delta_j\leq C\delta;
 $$
here we have also used \eqref{energ_bou}. The last inequality implies \eqref{tech_est}.

%{\color{red} It remains to show that each limit point $w$ is in $H^1(\Domain)$. To that end
%we may use a lower-semicontinuity argument as in \cite{Go}. Since this requires the auxiliary use
%of $SBV$-spaces, we refer to \cite{LN98} and \cite{Go} for details.
%By the lower-semicontinuity argument in  \cite{Go} for every $\Domain'\Subset\Domain$
%and every constant $K$ we have that $w\in SBV(\Domain')$ and
%$$
%\int_{\Domain'}|\nabla w|^2\dx+K{\mathcal H}^{d-1}(S(w)\cap\Domain')\le \liminf_{\e\to 0} F_\e^{k\e,r}(w_\e)\le C,
%$$
%where ${\mathcal H}^{d-1}$ is the $d-1$-dimensional Hausdorff measure and $S(w)$ is the set of discontinuity points
%of $w$. By the arbitrariness of $K$ and $\Domain'$ we then have that ${\mathcal H}^{d-1}(S(w))=0$ and
%$\int_{\Domain}|\nabla w|^2\dx<+\infty$; that is, $w\in H^1(\Domain)$.}

It remains to show that each limit point $w$ is in $H^1(\Domain)$. To that end we may use the `slicing technique' (see e.g.~\cite{LN98} Section 4.1, \cite{GCB} Chapter 15 or \cite{Handbook} Section 3.4). This general method allows to reduce the analysis to that of one-dimensional sections, and recover a lower bound by integrating over all sections. It has already been applied in \cite{Go} to sequences of nonlinear functionals of the form
\begin{eqnarray}\label{def-Fe-Go}
{1\over \e^{d+1}}\int_{\Domain}\int_{\Domain} a\Bigl({y-x\over\e}\Bigr)
f\Bigl({(u(y)-u(x))^2\over\e}\Bigr)dy \dx
\end{eqnarray}
in order to obtain compactness in spaces of functions with bounded variation. In our case we are in a simplified situation with $f$ equal the identity and we can improve the result to compactness in $H^1(\Domain)$.

In the one-dimensional case it is not restrictive to study functionals of the form
\begin{eqnarray}\label{def-Ge}
G_\e(u)&=&\int_{(0,1)}\int_{(-1,1)}
\Bigl({u(x+\e\xi)-u(x)\over\e}\Bigr)^2 d\xi \dx,
\end{eqnarray}
and regard all functions as defined on $\rr$. With Fatou's lemma in mind, in order to have a lower bound it suffices to examine
separately the functionals
\begin{eqnarray}\label{def-Gexi}
G^\xi_\e(u)&=&\int_{(0,1)}
\Bigl({u(x+\e\xi)-u(x)\over\e}\Bigr)^2 \dx
\end{eqnarray}
for fixed $\xi\in(-1,1)$.

For simplicity, we treat the case $\xi\in (0,1)$. We may suppose that $u_\e\to u$ in $L^2(\rr)$. Note that for almost all $t\in(0,1)$ the piecewise-constant functions
$u_{\e,\xi,t}$ defined by
$$
u_{\e,\xi,t}(x)= u_\e(\e\xi t+\e\xi k)\qquad \hbox{if } \e\xi k\le x< \e\xi (k+1)
$$
converge to $u$ in $L^2(\rr)$, and we have
\begin{eqnarray}\label{Gexi-1}\nonumber
G^\xi_\e(u_\e)&\ge&\sum_{k=1}^{\lfloor1/\e\xi\rfloor-1}\int_{k\e\xi}^{(k+1)\e\xi}
\Bigl({u_\e(x+\e\xi)-u_\e(x)\over\e}\Bigr)^2 \,dt
\\ \nonumber
&=&\sum_{k=1}^{\lfloor1/\e\xi\rfloor-1}\int_{0}^{1}\e\xi
\Bigl({u_\e((k+1)\e\xi+t\e\xi)-u_\e(k\e\xi+t\e\xi)\over\e}\Bigr)^2 \,dt
\\ \nonumber
&=&\xi^2\int_{(0,1)}\sum_{k=1}^{\lfloor1/\e\xi\rfloor-1}\e\xi
\Bigl({u_{\e,\xi,t}((k+1)\e\xi)-u_{\e,\xi,t}(k\e\xi)\over\e\xi}\Bigr)^2 \,dt=
\\
&\ge&\xi^2\int_{(0,1)}\int_{(\delta,1-\delta)}(u'_{\e,\xi,t}(x))^2\dx \,dt,
\end{eqnarray}
eventually for all $\delta>0$ fixed, where we have identified the discrete function $k\e\xi \mapsto u_{\e,\xi,t}(k\e\xi )$ defined on $\e\xi\ZZ$ with its piecewise-affine interpolation. Note that for almost all $t$ this functions still converge to $u$.
From \eqref{Gexi-1} we deduce that $u\in H^1(\delta,1-\delta)$. By the arbitrariness of $\delta$ and the uniformity of the bound on
the $L^2$-norm of $u'$ we deduce that $u\in H^1(0,1)$. For more details on this proof we refer to \cite{LN98}, where the nonlinear case is treated.

The deduction of the $d$-dimensional lower bound from the $1$-dimensional one can be obtained by repeating word for word the proof of \cite{LN98} Theorem 5.19 with $G^\xi_\e$ in the place of $F^1_\e$ in the notation therein.
This completes the proof of the compactness.
\end{proof}

\section{Poincar\'e inequalities}\label{s_poinc}

%\subsection{A Poincar\'e-Wirtinger inequality}

We first prove a Poincar\'e-Wirtinger inequality as follows.

\begin{theorem}[Poincar\'e-Wirtinger inequality]\label{t_poin}
Let $\Domain$ be a Lipschitz bounded domain.
%and assume that there is an isomorphism defined in an open
% neighbourhood of $\overline\Domain$ that maps $\Domain$ onto a convex open set $Q$.
For each fixed $r_0>0$ there exists a constant $C>0$ such that for any
$v\in L^2(\Domain)$ we have
\begin{equation}\label{poin_solid}
\int_\Domain (v(x)-v_\Domain)^2\,dx\leq C \int_{\Domain}\int_
{\{\xi:|\xi|\le r_0, x+\e\xi\in\Domain\}}
%{\begin{array}{c}\scriptstyle|\xi|\le r_0,\\[-4pt] \scriptstylex+\e\xi\in\Domain\end{array}}
{}
\Bigl({v(x+\e\xi)-v(x)\over\e}\Bigr)^2d\xi \dx,
\end{equation}
%where $r_0>0$ is defined in \eqref{ass-aa},
and $v_\Domain$ is the average of $v$ over $\Domain$. The constant $C$ does not depend on $\e$.
\end{theorem}
\begin{corollary}\label{co_poin} Let  $r_0>0$ be defined in \eqref{ass-aa}.
Let $k>0$ and $r>0$ be the same as in Theorem \ref{t_ext}. Then for any $u\in L^2(\Domain)$ the following inequality holds:
\begin{equation}\label{poi_modif}
\int_{\Domain(k\e)\cap\e E}\big(u(x)-u_{\{\Domain(k\e)\cap\e E\}}\big)^2\,dx\leq C F_\e(u);
\end{equation}
here
$$
u_{\{\Domain(k\e)\cap\e E\}}=\frac{1}{|\Domain(k\e)\cap\e E|}\int_{\{\Domain(k\e)\cap\e E\}} u(x)\,dx.
$$
\end{corollary}
\begin{proof}[Proof of Theorem {\rm\ref{t_poin}}] We set
$$
F_\e^0(r,v)=\int\limits_{\Domain}\int_
{\{\xi:|\xi|\le r,
x+\e\xi\in\Domain\}}
%{\begin{array}{c}\scriptstyle|\xi|\le r,\\[-4pt] \scriptstyle
%x+\e\xi\in\Domain\end{array}}
{}
\Bigl({v(x+\e\xi)-v(x)\over\e}\Bigr)^2d\xi \dx
$$
and
$$
F^1(G_1,G_2,v)=\int_{G_1}\int_{G_2}
\bigl({v(x)-v(y)}\bigr)^2 dx\,dy.
$$

In what follows the notation $\Domain^\e$ is used for $\frac1\e\Domain$.
\smallskip

%Since for any $M>0$
%$$
%\int\limits_{\begin{array}{c}\scriptstyle x\in \Domain^\e,\,y\in\Domain^\e\\[-4pt]
%\scriptstyle |x-y|<r \end{array}}(u(x)-u(y))^2\,dx dy=M^{-2d}
%\int\limits_{\begin{array}{c}\scriptstyle x\in M\Domain^\e,\,y\in M\Domain^\e\\[-4pt]
%\scriptstyle |x-y|<Mr \end{array}}(u(Mx)-u(My))^2\,dx dy=M^{-2d}
%\int\limits_{\begin{array}{c}\scriptstyle x\in \Domain^{\e/M},\,y\in \Domain^{\e/M}\\[-4pt]
%\scriptstyle |x-y|<Mr \end{array}}(u(x)-u(y))^2\,dx dy
%$$
%%%%%%
%\int\limits_{\begin{array}{c}\scriptstyle x\in \Domain^\e,\,y\in\Domain^\e\\[-4pt]
%\scriptstyle |x-y|<M \end{array}}(u(x)-u(y))^2\,dx dy.
%It can be  shown in exactly the same way as in the proof of Lemma \ref{ene_locali} that for any $M>0$
%there exists $C_M>0$ such that
% \begin{equation}\label{ineq_locpoi1}
%\int\limits_{\begin{array}{c}\scriptstyle x\in \Domain^\e,\,y\in\Domain^\e\\[-4pt]
%\scriptstyle |x-y|<r \end{array}}(u(x)-u(y))^2\,dx dy\leq c_M
%\int\limits_{\begin{array}{c}\scriptstyle x\in \Domain^\e,\,y\in\Domain^\e\\[-4pt]
%\scriptstyle |x-y|<M \end{array}}(u(x)-u(y))^2\,dx dy.
%\end{equation}
We first consider the case when $\Domain$ is a cube,  $\Domain=(-\frac L2,\frac L2)^d$, and $r$ is a sufficiently large number, say $r\geq3\sqrt{d}$.  We also assume that $L/\e$ is an integer number.

Denote $\mathcal{S}^\e=\{j\in\mathbb Z^d\,:\,j+[-\frac12,\frac12]^d\}\cap\Domain^\e\not=\emptyset$.
For any $i\in\mathcal{S}^\e $ and $j\in \mathcal{S}^\e$ construct a path $\boldsymbol{\gamma}(i,j)=\{j_k\}_{k=1}^N$ in $\mathbb Z^d$
such that $j_1=i$, $j_N=j$, $|j_k-j_{k+1}|=1$.  The path is constructed in such a way that it starts along
the first coordinate direction until the first coordinate of $j_k$ coincides with the first coordinate of $j$,
then it follows the second coordinate direction and so on.
We then have
\begin{itemize}
\item[{\it i.}]
the length of each path is not greater than $d\frac L\e$,
\item[{\it ii.}]  For each $j\in\mathcal{S}^\e$ the total number of paths $\{\gamma(i,l)\,:\,i,\,l\in\mathcal{S}^\e\}$
that pass through $j$ is not greater than $\big(\frac{L}{\e}\big)^{d+1}$:
\begin{equation}\label{path_num}
\#\big\{\boldsymbol{\gamma}(i,l)\,:\,i,\,l\in\mathcal{S}^\e,\ j\in\boldsymbol{\gamma}(i,l)\big\}\leq \Big(\frac{L}{\e}\Big)^{d+1}.
\end{equation}
\end{itemize}
For any $j\in \mathcal{S}^\e$ denote $Q_j=\e j+\e[-\frac12,\frac12]^d$.
For $i$ and $j$ in $\mathcal{S}^\e$  the ``interaction energy of the cubes $Q_i$ and $Q_j$'' can be estimated as follows.
We consider a path $\boldsymbol{\gamma}(i,j)$, denote the length of this path by $N$ and its elements by $\gamma_1,
\,\gamma_2,\ldots,\gamma_N$,
 and introduce the variables $\eta_2,\ldots,\eta_{N-1}$,  $\eta_k\in Q_0$.  Then we have
\begin{eqnarray*}&&
 \int_{\e Q_i}\int_{Q_j}\Big(\frac{u(x)-u(\e\xi)}{\e}\Big)^2\,d\xi dx\\
 &=&
 \e^{d-2}\int_{Q_0}\int_{Q_0}(u(\e\gamma_1+\e\eta_1)-u(\e\gamma_N+\e\eta_N))^2\,d\eta_1 d\eta_N
\\
&=&\e^{d-2}\int_{Q_0}\ldots\int_{Q_0}(u(\e\gamma_1+\e\eta_1)-u(\e\gamma_2+\e\eta_2)
+u(\e\gamma_2+\e\eta_2)-\ldots\\
&&\qquad\qquad\qquad\qquad\qquad\qquad\qquad -u(\e\gamma_N+\e\eta_N))^2\,d\eta_1d\eta_2\ldots d\eta_N
\\
&\leq& N\e^{d-2}\sum_{i-1}^{N-1}\int_{Q_0}\int_{Q_0}
(u(\e\gamma_i+\e\eta_i)-u(\e\gamma_{i+1}+\e\eta_{i+1})^2d\eta_id\eta_{i+1}\\
&\leq& (Ld)\e^{d-3}\sum_{i=1}^{N-1}\int_{Q_0}\int_{Q_0}
(u(\e\gamma_i+\e\xi)-u(\e\gamma_{i+1}+\e\eta)^2d\xi d\eta
\\
&\leq& (Ld)\e^{-3}\sum_{i=1}^{N-1}\int_{\e Q_0}\int_
{\{\xi: x+\e\xi\in\Domain,|\xi|<r\}}
%{\begin{array}{c}\scriptstyle x+\e\xi\in\Domain,\\[-4pt]
%\scriptstyle |\xi|<r \end{array}}
(u(\e\gamma_i+x)-u(\e\gamma_i+x+\e\xi)^2dx d\xi.
\end{eqnarray*}
Considering  \eqref{path_num} we deduce from the last inequality that
\begin{eqnarray*}
&&\int_\Domain\int_\Domain(u(x)-u(y))^2dx\,dy\\
&=&\sum_{i,\,l\in\mathcal{S}^\e}\e^{d+2}
 \int_{\e Q_i}\int_{Q_l}\Big(\frac{u(x)-u(\e\xi)}{\e}\Big)^2\,d\xi dx
\\
&\leq&(Ld)\e^{d-1}\Big(\frac L\e\Big)^{d+1}\sum_{j\in \mathcal{S}^\e}\int_{x\in\e Q_0}\int_
{\{\xi: x+\e\xi\in\Domain,|\xi|<r\}}
%{\begin{array}{c}\\[-15pt]\scriptstyle x+\e\xi\in\Domain,\\[-4pt]\scriptstyle |\xi|<r \end{array}}
(u(\e j+x)-u(\e j+x+\e\xi))^2dx d\xi
\\
&\leq& L^{d+2}d \int_{x\in\Domain}\int_
{\{\xi: x+\e\xi\in\Domain,|\xi|<r\}}
%{\begin{array}{c}\\[-15pt]\scriptstyle x+\e\xi\in\Domain,\\[-4pt]\scriptstyle |\xi|<r \end{array}}
\Big(\frac{u(x)-u(x+\e\xi}\e\Big)^2dx d\xi.
\end{eqnarray*}
Since
$$
\int_\Domain\int_\Domain (u(x)-u(y))^2dx\,dy=2\int_\Domain(u(x)-u_\Domain)^2dx,
$$
this yields the desires inequality in the case of a cubic domain.

The case of an arbitrary $r>0$ and $L>0$ can be reduced to the one just studied by standard
scaling arguments.

If $\Domain$ is a strongly star-shaped domain, then there exists a cube ${\bf B}$ and a Lipschitz isomorphism $J:\Domain\mapsto \mathbf{B}$ such that $|J(x)-J(y)|\leq \ell|x-y|$,  \, $\Big|\frac{\partial J}{\partial x}\Big|\leq \ell$,
$\Big|\Big(\frac{\partial J}{\partial x}\Big)^{-1}\Big|\leq \ell$ for some $\ell>0$.
For an arbitrary  $u\in L^2(\Domain)$ denote $u_J(x)=u(J^{-1}(x))$ and $u_{{\bf B},J}=\int_{\bf B}u_J(x)dx$.
Also, we set $r_1=r/\ell$.
Since the desired inequality has been proved for cubic domains,  we have
\begin{eqnarray*}
&&\int_\Domain\int_\Domain(u(x)-u(y))^2\,dx\,dy\\
&=&\int_{\bf B}\int_{\bf B}(u_J(x)-u_J(y))^2\Big|\frac{\partial J^{-1}}{\partial x}(x)\Big|\,
\Big|\frac{\partial J^{-1}}{\partial x}(y)\Big|\,dx\,dy
\\
&\leq& \ell^2 \int_{\bf B}\int_{\bf B}(u_J(x)-u_J(y))^2\,dx\,dy
\\
&\leq&  C\e^{-d}\ell^2\int_{\bf B}
\int_{\{y\in{\bf B}: |y-x|<\e r_1 \}}
%{\begin{array}{c}\\[-15pt]\scriptstyle y\in{\bf B},\\[-4pt]\scriptstyle |y-x|<\e r_1 \end{array}}
\Big(\frac{u_J(x)-u_J(y)}\e\Big)^2\,dy\,dx
\\
&\leq& C\e^{-d}\ell^2\int_{ \Domain}
\int_{\{\xi: x+\e\xi\in\Domain,|\xi|<r\}}
%{\begin{array}{c}\\[-15pt]\scriptstyle y\in{\Domain},\\[-4pt]\scriptstyle |y-x|<\e r \end{array}}
\Big(\frac{u(x)-u(y)}\e\Big)^2\Big|\frac{\partial J}{\partial x}(x)\Big|\,
\Big|\frac{\partial J}{\partial x}(y)\Big|\,dy\,dx
\\
&\leq&  C\e^{-d}\ell^4\int_{ \Domain}
\int_{\{\xi: x+\e\xi\in\Domain,|\xi|<r\}}
%{\begin{array}{c}\\[-15pt]\scriptstyle y\in{\Domain},\\[-4pt]\scriptstyle |y-x|<\e r \end{array}}
\Big(\frac{u(x)-u(y)}\e\Big)^2\,dy\,dx,
\end{eqnarray*}
where the constant $C$ depends only on the size of ${\bf B}$, $r_1$ and $d$.

\smallskip
It remains to consider an arbitrary bounded Lipschitz set $\Domain$. Such a set can be represented
as a union of a finite number of strongly star shaped domains, we denote these domains $\Domain_1,\ldots,\Domain_N$.

We first consider the case $N=2$, we denote by $\widetilde{\bf B}$ a cube such that $\widetilde{\bf B}\subset\Domain$,
$|\widetilde{\bf B}\cup\Domain_1|\geq\frac12 |\widetilde{\bf B}|$,
$|\widetilde{\bf B}\cup\Domain_2|\geq\frac12 |\widetilde{\bf B}|$. Notice that
$|\widetilde{\bf B}\cup\Domain_1|
=|\widetilde{\bf B}\cup\Domain_2|=\frac12 |\widetilde{\bf B}|$ if the interiors of $\Domain_1$ and $\Domain_2$ do not
intersect.
%The existence of a cube with these properties is evident.
In the rest of the proof the symbols  $\widetilde{\bf B}_1$ and $\widetilde{\bf B}_2$  stand for $\widetilde{\bf B}\cup\Domain_1$ and
$\widetilde{\bf B}\cup\Domain_2$, respectively.

If we denote
$$
\overline{u}_k=\frac1{|\Domain_k|}\int\limits_{\Domain_k}u(x)\,dx,\ \  k=1,2;\quad\overline{u}_{0,k}=\frac1{|\widetilde{\bf B}_k|}\int\limits_{\widetilde{\bf B}_k}u(x)\,dx,\ \  k=1,2;\quad
\overline{u}_{0}=\frac1{|\widetilde{\bf B}|}\int\limits_{\widetilde{\bf B}}u(x)\,dx
%;\qquad \overline{u}_0=
%\frac1{|\widetilde{\bf B}|}\int_{\widetilde{\bf B}}u(x)\,dx.
$$
then
\begin{eqnarray*}
(\overline{u}_1-\overline{u}_{0,1})^2&=&\Big(\frac1{|\widetilde{\bf B}_1|\,|\Domain_1|}\int_{\widetilde{\bf B}_1}
\int_{\Domain_1}u(x)\,dx\,dy-
\frac1{|\widetilde{\bf B}_1|\,|\Domain_1|}\int_{\widetilde{\bf B}_1}\int_{\Domain_1}u(y)\,dx\,dy\Big)^2
\\
&\leq& \frac1{|\widetilde{\bf B}_1|\,|\Domain_1|}\int_{\widetilde{\bf B}_1}\int_{\Domain_1}
(u(x)-u(y))^2\,dx\,dy\\
&\leq& \frac1{|\widetilde{\bf B}_1|\,|\Domain_1|}\int_{\Domain_1}
\int_{\Domain_1}(u(x)-u(y))^2\,dx\,dy\
\\
&\leq&
C\e^{-d}\int_{ \Domain_1}
\int_{\{y\in{\Domain_1}: |y-x|<\e r \}}
%{\begin{array}{c}\\[-15pt]\scriptstyle y\in{\Domain_1},\\[-4pt]\scriptstyle |y-x|<\e r \end{array}}
\Big(\frac{u(x)-u(y)}\e\Big)^2\,dy\,dx\\
&\leq& C\e^{-d}\int_{ \Domain}
\int_{\{ y\in{\Domain}: |y-x|<\e r\}}
%{\begin{array}{c}\\[-15pt]\scriptstyle y\in{\Domain},\\[-4pt]\scriptstyle |y-x|<\e r \end{array}}
\Big(\frac{u(x)-u(y)}\e\Big)^2\,dy\,dx;
\end{eqnarray*}
here we have used inequality \eqref{poin_solid} in $\Domain_1$ that holds because $\Domain_1$ is a strongly star shaped domain. In the same way we prove that
$$
(\overline{u}_{0,1}-\overline{u}_{0,2})^2
\leq C\e^{-d}\int_{ \Domain}
\int_{\{ y\in{\Domain}: |y-x|<\e r \}}
%{\begin{array}{c}\\[-15pt]\scriptstyle y\in{\Domain},\\[-4pt]\scriptstyle |y-x|<\e r \end{array}}
\Big(\frac{u(x)-u(y)}\e\Big)^2\,dy\,dx,
$$
and
$$
(\overline{u}_{0,2}-\overline{u}_{2})^2
\leq C\e^{-d}\int_{ \Domain}
\int_{\{ y\in{\Domain}: |y-x|<\e r \}}
%{\begin{array}{c}\\[-15pt]\scriptstyle y\in{\Domain},\\[-4pt]\scriptstyle |y-x|<\e r \end{array}}
\Big(\frac{u(x)-u(y)}\e\Big)^2\,dy\,dx.
$$
Therefore,
$$
(\overline{u}_{1}-\overline{u}_{2})^2
\leq C\e^{-d}\int_{ \Domain}
\int_{\{ y\in{\Domain}: |y-x|<\e r \}}
%{\begin{array}{c}\\[-15pt]\scriptstyle y\in{\Domain},\\[-4pt]\scriptstyle |y-x|<\e r \end{array}}
\Big(\frac{u(x)-u(y)}\e\Big)^2\,dy\,dx.
$$
Since $u_\Domain\in\big(\overline{u}_{1},\overline{u}_{2}\big)$,  the last inequality yields
\begin{eqnarray*}
\int_\Domain(u(x)-u_\Domain)^2\,dx
&\leq& \sum_{k=1}^2 \
\Big(2\int_{\Domain_k}(u(x)-\overline{u}_{k})^2\,dx+2|\Domain_k|(\overline{u}_{k}-u_\Domain)^2\Big)
\\
&\leq& 2\sum_{k=1}^2\int_{\Domain_k}(u(x)-\overline{u}_{k})^2\,dx+
2|\Domain|(\overline{u}_{1}-\overline{u}_2)^2\\
&\leq& C\e^{-d}\int_{ \Domain}
\int_{\{ y\in{\Domain}: |y-x|<\e r \}}
%{\begin{array}{c}\\[-15pt]\scriptstyle y\in{\Domain},\\[-4pt]\scriptstyle |y-x|<\e r \end{array}}
\Big(\frac{u(x)-u(y)}\e\Big)^2\,dy\,dx.
\end{eqnarray*}

The case $N>2$ can be achieved by induction.
\end{proof}

We next consider functions with given boundary data.
\begin{lemma}[Poincar\'e inequality]
Let $\Domain$ be a bounded set and let $u\in L^2(\Domain)$ be such that $u=0$ on a $2\e$-neighbourhood of $\partial \Domain$
(and extended to $0$ outside $\Domain$). Then there exists a constant $C$ depending only on the diameter of $\Domain$ such
that
\begin{equation}
\int_\Domain|u(x)|^2\dx\le C{1\over \e^{d+2}}\int_\Domain\int_{\{|\xi|\le \e\}}(u(x+\xi)- u(x))^2 d\xi dx\,.
\end{equation}
\end{lemma}

\begin{proof}
It suffices to treat the case $d=1$ and $\Domain=(0,1)$, the general case being recovered from this one by considering one-dimensional stripes.
For notational convenience we replace $\e$ by $2\e$, so that our claim becomes that
\begin{equation}\label{po-one}
\int_0^1|u(x)|^2\dx\le C{1\over \e^{3}}\int_{-\infty}^{+\infty}\int_{x-2\e}^{x+2\e}(u(y)- u(x))^2 dy\dx\,,
\end{equation}
keeping in mind that the first integral in the right-hand side is indeed restricted to $(0,1)$.

For all $k\in\NN$ we note that, since
$$
(x-2\e,x+2\e)\supset (k\e-\e,k\e+\e) \hbox{ if } x\in (k\e-\e,k\e+\e),
$$
we have
\begin{eqnarray}\nonumber
&&\int_{k\e-\e}^{k\e+\e}\int_{x-2\e}^{x+2\e}(u(y)- u(x))^2 dy\dx\\
&\ge&\nonumber
\int_{k\e-\e}^{k\e+\e}\int_{k\e-\e}^{k\e+\e}(u(y)- u(x))^2 dy\dx\\
&\ge&\nonumber
\int_{k\e-\e}^{k\e}\int_{k\e}^{k\e+\e}(u(y)- u(x))^2 dy\dx\\
&=&\nonumber
\e\int_{k\e-\e}^{k\e}|u(x)|^2\dx-2\int_{k\e-\e}^{k\e}u(x)\dx
\int_{k\e}^{k\e+\e}u(y) dy+ \e \int_{k\e}^{k\e+\e}|u(y)|^2 dy
\\
&=&\nonumber
\e\Bigl(\int_{k\e-\e}^{k\e}|u(x)|^2\dx-2\sqrt{\int_{k\e-\e}^{k\e}|u(x)|^2\dx}
\sqrt{\int_{k\e}^{k\e+\e}|u(y)|^2 dy}+ \int_{k\e}^{k\e+\e}|u(y)|^2 dy\Bigr)
\\
&=&
\e\Biggl(\sqrt{\int_{k\e-\e}^{k\e}|u(x)|^2\dx}-
\sqrt{\int_{k\e}^{k\e+\e}|u(y)|^2 dy}\Biggr)^2.
\end{eqnarray}

Note that for $k=0$ this gives
\begin{eqnarray}\nonumber
\int_{0}^{\e}|u(y)|^2\,dy\le{1\over\e}\int_{-\e}^{\e}\int_{x-2\e}^{x+2\e}(u(y)- u(x))^2 dy\,dy\,
\end{eqnarray}
By a recursive argument from $k=0$ we deduce that
\begin{eqnarray}\nonumber
\int_{k\e}^{k\e+\e}|u(y)|^2\,dy&\le&{1\over\e}\Biggl(\sum_{j=0}^k\sqrt{\int_{j\e-\e}^{j\e+\e}\int_{x-2\e}^{x+2\e}(u(y)- u(x))^2 dy\,dx}\Biggr)^2\\ \nonumber
&\le&{1\over\e^2}\sum_{j=0}^k\int_{j\e-\e}^{j\e+\e}\int_{x-2\e}^{x+2\e}(u(y)- u(x))^2 dy\,dx
\\ \nonumber
&\le&{2\over\e^2}\int_{-\infty}^{+\infty}\int_{x-2\e}^{x+2\e}(u(y)- u(x))^2 dy\,dx,
\end{eqnarray}
where the factor $2$ takes into account that the intervals $({j\e-\e},{j\e+\e})$ overlap for consecutive values of $j$.
Noting that indeed the term with $k=0$ is $0$ by our assumptions on the values of $u$ close to the boundary,
it suffices now to sum up the contribution over all $k\in\{1,\ldots, \lfloor1/\e\rfloor\}$ to obtain
\begin{eqnarray}\nonumber
\int_{0}^{1}|u(y)|^2\,dy
&\le&2{\lfloor1/\e\rfloor\over\e^2}\int_{-\infty}^{+\infty}\int_{x-2\e}^{x+2\e}(u(y)- u(x))^2 dy\,dx,
\end{eqnarray}
which gives \eqref{po-one} with $C=2$.
Note that if the interval $(0,1)$ is substituted by any interval then we can take $C$ as twice the length of the interval.
\end{proof}

\section{Definition of the homogenized energy density}\label{s_hom_formula}

Let $b$ be as in Section \ref{setting_sec}.
For all $K\in\NN$ we set
\begin{equation}
b^\omega_K(x,y)=\begin{cases} b^\omega(x,y) & \hbox{ if }|x-y|<K\cr
0 & \hbox{otherwise,}\end{cases}
\end{equation}
and, for $z\in \rr^d$, $U$ open subset of $\rr^d$, and $K\in\NN$ we define
\begin{equation}\label{MKzU}
{\cal M}^\omega_K(z,U)=\inf\Bigl\{\int_U\int_{\rr^d}b^\omega_K(x,y)(v(x)-v(y))^2dx\,dy:
v(x)=\langle z,x\rangle \hbox{ if  dist}(x,\partial U)<K\Bigr\}.
\end{equation}
Note that, using $v(x)=\langle z,x\rangle$ as a test function, we get
\begin{equation}\label{MKzU-estimate}
{\cal M}^\omega_K(z,x+Q_R)\le C R^d|z|^2
\end{equation}
for all $x$ and $R$.

\begin{lemma}\label{lemma-one}
For all $K$ and $z$ the limit
\begin{equation}\label{limMKzU}
\gamma_K(z)=\lim_{R\to+\infty}{{\cal M}^\omega_K(z,Q_R)\over R^d}
\end{equation}
exists almost surely, it is independent of $\omega$, and $K\mapsto \gamma_K(z)$ is an increasing function.
Moreover, there exists an increasing function $f_K$ with
$$
\lim_{R\to+\infty} f_K(R)=+\infty
$$
such that
\begin{equation}\label{limMKzUbis}
\gamma_K(z)=\lim_{R\to+\infty}{{\cal M}^\omega_K(z,x_R+Q_R)\over R^d}
\end{equation}
for all $\{x_R\}$ such that $|x_R|\le R f_K(R)$.
\end{lemma}
\begin{proof}
Our arguments rely on a uniform version of  the sub-additive ergodic theorem, see \cite[Theorem 1]{KrPy}.
For any $j\in \mathbb Z^{d,+}=\{0,\,1,\,2,\ldots\}^d$ we define $Q^j=j+\bar{\frac12}+Q$, where $\bar{\frac12}$
is the vector $(\frac12,\frac12,\ldots,\frac12)$.
For any finite subset $\mathcal{A}$ of $\mathbb Z^{d,+}$ denote $Q^{\mathcal{A}}=\bigcup_{j\in \mathcal{A}}
Q^j$, and $\Phi_K(z,\mathcal{A})=\mathcal{M}_K^\omega(z,Q^{\mathcal{A}})$.

From definition \eqref{MKzU} for any non-intersecting finite sets  $\mathcal{A}$ and $\mathcal{B}$ we have
$$
\Phi_K(z,\mathcal{A}\cup\mathcal{B})\leq \Phi_K(z,\mathcal{A})+\Phi_K(z,\mathcal{B}).
$$
Since $b_K^\omega(x,y)$ is statistically homogeneous, the family $\{ \Phi_K(z,\mathcal{A})\}$ is stationary;
that is, for any $j\in\mathbb Z^{d,+}$ and any finite collection $\mathcal{A}_1,\ldots, \mathcal{A}_N$
the joint law of   $\{ \Phi_K(z,\mathcal{A}_1+j),\ldots,  \Phi_K(z,\mathcal{A}_N+j)\}$ is the same as the joint law
of $\{ \Phi_K(z,\mathcal{A}_1),\ldots,  \Phi_K(z,\mathcal{A}_N)\}$.
Then according to Theorem 1 in \cite{KrPy} there exists $\gamma_K(z)$ such that for any $N>0$ we have
\begin{equation}\label{ave_bir}
  \lim\limits_{R\to\infty}\sup\Big\{\Big|{{\cal M}^\omega_K\big(z,R(x+Q)\big)\over R^d}-\gamma_K(z)\Big|: {|x|\leq N} \Big\}=0
\end{equation}
 almost surely.  This implies \eqref{limMKzU}; moreover, since $b^\omega>0$, $K\mapsto \gamma_K(z)$ is an increasing function.

Note that we can choose a (slowly growing) sequence $N=N^\omega(R)$ such that \eqref{ave_bir} still holds, which yields \eqref{limMKzUbis}.
\end{proof}

\begin{definition}[homogenized energy function]
We define
$$
\gamma(z)=\lim_{K\to+\infty}\gamma_K(z)=\sup_{K>0}\gamma_K(z).
$$
\end{definition}

For $z\in \rr^d$, $U$ open subset of $\rr^d$, and $K\in\NN$ we set
\begin{equation}\label{MtKzU}
\widetilde{\cal M}^\omega_K(z,U)=\inf\Bigl\{\int_U\int_{U}b^\omega(x,y)(v(x)-v(y))^2dx\,dy:
v(x)=\langle z,x\rangle \hbox{ if  dist}(x,\partial U)<K\Bigr\}
\end{equation}

Note that $\widetilde{\cal M}^\omega_K(z,U)$ cannot be directly compared with ${\cal M}^\omega_K(z,U)$ as defined in \eqref{MKzU} since on one side $b^\omega_K\le b^\omega$ while the second integral is performed on $U$ and not $\rr^d$.
However, still using $v(x)=\langle z,x\rangle$ as a test function, we get
\begin{equation}\label{MKtzU-estimate}
\widetilde{\cal M}^\omega_K(z,x+Q_R)\le C R^d|z|^2
\end{equation}
for all $x$ and $R$.

\begin{lemma}\label{lemma-two}
Let $b^\omega$ be coercive.
For all $K$ and $z$ we have
\begin{eqnarray}\label{limMtKzU}
\gamma(z)=\lim_{K\to+\infty}\limsup_{R\to+\infty}{\widetilde{\cal M}^\omega_K(z,Q_R)\over R^d}
=\lim_{K\to+\infty}\liminf_{R\to+\infty}{\widetilde{\cal M}^\omega_K(z,Q_R)\over R^d}
\end{eqnarray}
 almost surely.
\end{lemma}

The proof of this lemma is based on the following proposition.

\begin{proposition}\label{prepro1}
If $U$ is a cube in $\rr^d$ and $v\in L^2(U)$ then we have
\begin{equation}\label{prelemma}
\int_{\{x,y\in U: |x-y|>K\}}b^\omega(x,y) (v(x)-v(y))^2dx\,dy
\le C K^{-\kappa}\int_{\{x,y\in U: |x-y|<1\}} (v(x)-v(y))^2dx\,dy\,,
\end{equation}
with $C$ depending only on the bounds on $b^\omega$ and the dimension $d$.
\end{proposition}

%\begin{proposition} \label{prepro2} We have
%\begin{equation}\label{prelemma-2}
%|\widetilde{\cal M}^\omega_K(z,U)-{\cal M}^\omega_K(z,U)|\le CK^{-\kappa}\widetilde{\cal M}^\omega_K(z,U)
%+C|z|^2 K^{1-\kappa}{\cal H}^{d-1}(\partial U)
%\end{equation}
%\end{proposition}

\medskip
\begin{proof}[Proof of Proposition {\rm \ref{prepro1}}]
Without loss of generality we may assume that the cube $U$ is centered at the origin; i.e., $U=Q_T$ for some $T>0$.
Furthermore, we may suppose that $T$ is integer, and cover $Q_T$ with the set of unit cubes $Q(j)=Q+j$, $j\in\mathbb Z^d\cap U$.
If $K>T$, the statement trivially holds. Otherwise, for any $j'$ and $j''$  such that $|j'-j''|_1=n$ with $n\geq K$
we consider a path (i.e., an array of points in $\ZZ^d$), $j'=j_0,j_1,\ldots, j_n=j''$, with $|j_i-j_{i+1}|_1=1$, that has the following properties: in the starting segment of this path $j_0,j_1,\ldots, j_{n_1}$ only the first coordinate is changed
 until it is equal to the first coordinate of $j''$ (i.e., $n_1=j''_1-j'_1$, and $j_{i+1}=j_{i}+(1,0,\ldots,0)$). Then we proceed
with the second coordinate, and so on.

In order to estimate the contribution to the energy of the interaction between the cubes $Q(j')$ and $Q(j'')$, with fixed $n$ we first estimate the integral
\begin{eqnarray*}
&&\int_{\{(y_0,y_n)\in Q\times  Q\}}(v(y_0+j_0)-v(y_n+j_n))^2 dy_0dy_n
\\
&=&\int_Q\ldots\int_Q
\Big(\sum\limits_{i=0}^{n-1}\big (v(y_i+j_i)-v(y_{i+1}+j_{i+1})\big)\Big)^2 dy_0dy_1\ldots dy_n
\\
&\leq& n \int\limits_Q\int\limits_Q
\sum\limits_{i=0}^{n-1}\big (v(x+j_i)-v(y+j_{i+1})\big)^2 dx\,dy\,.
\end{eqnarray*}

Note that each pair
of neighbouring points in $U\cap\mathbb Z^d$ belongs to not more than $n^d$ paths as described above
for some pair $j',\,j''$ in $U$ such that $|j'-j''|_1=n$. Taking this into account and summing up over all $j',\,j''$ in $U\cap\mathbb Z^d$  with $|j'-j''|_1=n$ we obtain
$$
\sum\limits_{\begin{array}{c}\scriptstyle j',\,j'' \in U\cap \mathbb Z^d\\[-1mm] \scriptstyle
|j'-j''|\big._1=n \end{array}}\ \int_{Q\times  Q}(v(x+j')-v(y+j''))^2 dx\,dy\leq n^{d+1}\int_{
(U\times U)\cap \{|x-y|_1\leq 2\}}(v(x)-v(y))^2dx\,dy.
$$
Taking \eqref{ass-a_bis} into account, we have
\begin{eqnarray*}
&&\int_{\{(x,y)\in U\times U: |x-y|>K\}}b^\omega(x,y) (v(x)-v(y))^2dx\,dy
\\
&\leq& C \sum_{n=K}^T \frac{n^{d+1}}{(1+n)^{d+2+\kappa}}\int_{
\{(x,y)\in U\times U: |x-y|_1\leq 2\}}(v(x)-v(y))^2dx\,dy
\\
&\leq& CK^{-\kappa}\int_{
\{(x,y)\in U\times U: |x-y|_1\leq 2\}}(v(x)-v(y))^2dx\,dy.
\end{eqnarray*}
The desired statement follows from the last inequality by a scaling argument.
\end{proof}

\begin{proof}[Proof of Lemma {\rm \ref{lemma-two}}]
Denote
\begin{equation}\label{MKzU-2}
\overline{\cal M}^\omega_K(z,U)=\inf\Bigl\{\int_U\int_{U}b^\omega_K(x,y)(v(x)-v(y))^2dx\,dy:
v(x)=\langle z,x\rangle \hbox{ if  dist}(x,\partial U)<K\Bigr\}.
\end{equation}
Then
\begin{eqnarray}\label{MKzU-3}\nonumber
0&\leq& {\cal M}^\omega_K(z,U)-\overline{\cal M}^\omega_K(z,U)=
\int_U\int_{\mathbb R^d\setminus U}b^\omega_K(x,y)\langle z,(x-y)\rangle^2dx\,dy\\
&\leq& C|z|^2 K^{1-\kappa}{\cal H}^{d-1}(\partial U).
\end{eqnarray}

Let $u$ be a minimizer for ${\cal M}^\omega_{2K}(z,U)$ (which we may assume exists). Let $v$ be given by Definition \ref{coerciveness}
with $\Xi=2K$.
We then have
\begin{eqnarray}\label{ert}\nonumber
\widetilde{\cal M}^\omega_{K}(z,U)&\le& \int_U\int_{U}b^\omega(x,y)(v(x)-v(y))^2dx\,dy\\ \nonumber
&=&
\int_U\int_{U}b^\omega_{2K}(x,y)(v(x)-v(y))^2dx\,dy\\
\nonumber
&&\qquad+
\int_{\{x,y\in U: |x-y|>2K\}}b^\omega(x,y)(v(x)-v(y))^2dx\,dy
\\ \nonumber
&\le&
\overline{\cal M}^\omega_{2K}(z,U)
+C K^{-\kappa}\int_{\{x,y\in U: |x-y|<1\}} (v(x)-v(y))^2dx\,dy
\\ \nonumber
&\le&
\overline{\cal M}^\omega_{2K}(z,U)
+C K^{-\kappa}\int_{U\times U} b(x,y)(v(x)-v(y))^2dx\,dy
\\ \nonumber
&\le&
\overline{\cal M}^\omega_{2K}(z,U)
+C K^{-\kappa}|z|^2|U|
\\
&\le&
{\cal M}^\omega_{2K}(z,U)
+C K^{-\kappa}|z|^2|U|+C|z|^2 K^{1-\kappa}{\cal H}^{d-1}(\partial U),
\end{eqnarray}

Conversely, since $\overline{\cal M}^\omega_{K}(z,U)\le \widetilde{\cal M}^\omega_{K}(z,U)$
we have
\begin{equation}\label{ter}
{\cal M}^\omega_{K}(z,U)\le \widetilde{\cal M}^\omega_{K}(z,U)+ C|z|^2 K^{1-\kappa}{\cal H}^{d-1}(\partial U)
\end{equation}

Dividing by $R^d$, taking the upper limit in \eqref{ert} and the lower limit in \eqref{ter} with $U=Q_R$ we obtain
\begin{eqnarray*}\label{etr-2}
\gamma_K(z)=
\liminf_{R\to+\infty}{{\cal M}^\omega_{K}(z,Q_R)\over R^d}& \le& \liminf_{R\to+\infty}{\widetilde{\cal M}^\omega_{K}(z,Q_R)\over R^d}\\
&\le &\limsup_{R\to+\infty}{\widetilde{\cal M}^\omega_{K}(z,Q_R)\over R^d}\\
&\le&
\limsup_{R\to+\infty}{{\cal M}^\omega_{2K}(z,Q_R)\over R^d}+C K^{-\kappa}|z|^2\\
&=&\gamma_{2K}(z)+ C K^{-\kappa}|z|^2\end{eqnarray*}
Taking the limit as $K\to+\infty$ we obtain the claim.
%
%Notice that
%$$
%\overline{\cal M}^\omega_K(z,U)\leq \widetilde{\cal M}^\omega_K(z,U).
%$$
%Assuming that $\overline{\cal M}^\omega_K(z,U)< \widetilde{\cal M}^\omega_K(z,U)$, consider a function
%$u\in L^2(U)$ such that $u=<x,z>$ if $\mathrm{dist}(x,\partial U)<K$ and
%$$
%\widetilde{\cal M}^\omega_K(z,U)-\int_U\int_{U}b^\omega_K(x,y)(u(x)-u(y))^2dx\,dy\geq
%\frac12\big(\widetilde{\cal M}^\omega_K(z,U)-\overline{\cal M}^\omega_K(z,U)\big).
%$$
%Then
%$$
%\widetilde{\cal M}^\omega_K(z,U)-\overline{\cal M}^\omega_K(z,U)\leq 2
%\int_U\int_{U}b^\omega(x,y)(u(x)-u(y))^2dx\,dy-2\int_U\int_{U}b^\omega_K(x,y)(u(x)-u(y))^2dx\,dy
%$$
%$$
%=2\int\limits_{\{x,y\in U: |x-y|>K\}}b^\omega(x,y)(u(x)-u(y))^2dx\,dy
%\le C K^{-\kappa}\int\limits_{\{x,y\in U: |x-y|<1\}} (u(x)-u(y))^2dx\,dy;
%$$
%the last inequality here follows from Proposition \ref{prepro1}.
\end{proof}

\section{Homogenization}\label{hom}
We now state and prove a homogenization result with respect to the strong $L^2$-convergence.

\begin{theorem}\label{homthm}
Let $\Domain$ be an open set with Lipschitz boundary, and let $F^\omega_\e$ be given by
\eqref{def-Fe-0} on $L^2(\Omega)$.
Then $F^\omega_\e$ almost surely $\Gamma$-converge with respect to the $L^2$-convergence
to the functional
\begin{equation}\label{homfun}
F_{\rm hom}(u)=\int_\Domain \langle A_{\rm hom}\nabla u,\nabla u\rangle\dx
\end{equation}
on $H^1(\Domain)$, where $A_{\rm hom}$ is a symmetric matrix which satisfies
\begin{equation}\label{cpformula}
 \langle A_{\rm hom}z,z\rangle=\gamma(z).
 \end{equation}
\end{theorem}

The proof of this theorem will make use of a `convolution version' of a classical lemma by De Giorgi that allow to match the boundary values of a target function (see \cite{2018BP})

\begin{proposition}[treatment of boundary values]\label{bvp-v2}
Let $A$ be a bounded open set with Lipschitz boundary, let $v_\eta\to v$ in $L^2(A)$ with $v\in H^1(A)$.
%, and let
%$w_\eta$ converge weakly to $v$ in $H^1(A)$ and
%satisfy
%$$
%\limsup_{\eta\to 0}\int_{U}\int_{\{\xi: x+\eta\xi\in U\}} {1\over(1+|\xi|)^{d+2+\kappa}}
%\Bigl({w_\eta(x+\eta\xi)-w_\eta(x)\over\eta}\Bigr)^2d\xi \dx\le \omega (|U|)
%$$
%for all open sets $U$, where $\omega(t)\to 0$ as $t\to 0$.
For every $\delta>0$ there exist $v^\delta_\eta$ converging to $v$ in $L^2(A)$ such that
$$
v^\delta_\eta= v \hbox{ in } A\setminus A(\delta), \qquad v^\delta_\eta= v_\eta \hbox{ in } A(2\delta)
$$
and
$$
\limsup_{\eta\to 0} (F^\omega_\eta (v^\delta_\eta)- F^\omega_\eta(v_\eta))\le o(1)
$$
as $\delta\to 0$.
\end{proposition}

\begin{proof}[Proof of Theorem {\rm\ref{hom}}] By Remark \ref{co-co} it suffices to describe the $\Gamma$-limit in $H^1(D)$.

We note that $F^\omega_\e$ are quadratic functionals, so that also their $\Gamma$-limit is a quadratic functional (see \cite{GCB}). Then, if we prove that the $\Gamma$-limit exists and admits the representation
\begin{equation}\label{limwg}
F_{\rm hom}(u)=\int_\Domain \gamma(\nabla u)\dx,
\end{equation}
then also $\gamma$ must be a quadratic form on $\rr^d$, from which
the existence of a matrix $A_{\rm hom}$ satisfying \eqref{cpformula} follows.

We now prove \eqref{limwg}, first showing a lower bound.
We fix $\omega$, $u\in H^1(D)$ and a sequence $u_\e\to u$ with bounded $F_\e(u_\e)$.
As in \cite{2018BP}, we use a variation of the Fonseca-M\"uller blow-up technique \cite{FM}.
We first define the measures on $D$ given by
$$
\mu_\e(A)={1\over\e^{d+2}}\int_{A}\int_{D}
b^\omega\Bigl({x\over\e},{y\over\e}\Bigr)(u_\e(y)-u_\e(x))^2d\xi\dx.
$$
Since $\mu_\e(D)=F_\e(u_\e)$, these measures are equibounded, and we may suppose that they converge weakly$^*$ to some measure $\mu$. We now fix an arbitrary Lebesgue point $x_0$ for $u$ and $\nabla u$, and set $z=\nabla u(x_0)$. The lower-bound inequality is proved if we show that
\begin{equation}
{d\mu\over dx}(x_0)\ge \gamma(z).
\end{equation}

Upon a translation argument it is not restrictive to suppose that $x_0$ be a Lebesgue point of all $u_\e$ (upon passing to a subsequence), and that $u_\e(x_0)=u(x_0)=0$.
We note that for almost all $\rho>0$ we have $\mu_\e(x_0+Q_\rho)\to \mu(x_0+Q_\rho)$.
Since
$$
{d\mu\over dx}(0)= \lim_{\rho\to 0^+}{\mu(x_0+Q_\rho)\over\rho^d},
$$
and for almost all $\rho>0$
$$
\mu(Q_\rho)=\lim_{\e\to0}\mu_\e(x_0+Q_\rho)
$$
we may choose (upon passing to a subsequence) $\rho=\rho_\e$ with $1>\!>\rho>\!>\e$ such that
$$
{d\mu\over dx}(0)= \lim_{\e\to 0^+}{\mu_\e(x_0+Q_\rho)\over\rho^d}.
$$
Note that we may choose $\rho_\e$ tending to zero ``arbitrarily slow''; i.e.,
for all $f$ with $\lim\limits_{\e\to 0}f(\e)=0$ we may choose $\rho_\e$ with
\begin{equation}\label{slogro}
\rho_\e\ge f(\e).
\end{equation}
Note moreover that
\begin{eqnarray*}
\mu_\e(x_0+Q_\rho)&=&{1\over\e^d}\int_{x_0+Q_\rho}\int_{D}
b^\omega\Bigl({x\over\e},{y\over\e}\Bigr)\Bigl({u_\e(y)-u_\e(x)\over \e}\Bigr)^2\dx\,dy
\\
&\ge&{1\over\e^d}\int_{x_0+Q_\rho}\int_{x_0+Q_\rho}
b^\omega\Bigl({x\over\e},{y\over\e}\Bigr)\Bigl({u_\e(y)-u_\e(x)\over \e}\Bigr)^2\dx\,dy.
\end{eqnarray*}
%In order to ease the notation we will directly suppose that
%\begin{equation}{\rho\over\e}\in \NN, \end{equation}
%so that our claim is proven if we show that
%\begin{equation}\label{main_claim}
%\lim_{\e\to0^+}
%{1\over \rho^d}\int_{Q_\rho\cap \e E}\int_{\{\xi: x+\e\xi\in Q_\rho\cap \e E\}}
%b^\omega\Bigl({x\over\eta}, {y\over\eta}\Bigr)\Bigl({u_\e(x+\e\xi)-u_\e(x)\over \e}\Bigr)^2d\xi\dx
%\ge \langle A_{\rm hom}z,z\rangle.
%\end{equation}

We now change variables and set
$$
v_\e(y)= {u_\e(x_0+\rho y)\over \rho} \hbox{ for } y\in Q_1\,.
$$
Note that, since ${u(\rho y)\over \rho}$ converges to $\langle z,y\rangle$ as $\rho\to 0$ as we have assumed that $u(x_0)=0$,
and we also have assumed that  $u_\e(x_0)=0$, we may choose $\rho=\rho_\e$ above so that
$$
v_\e \to \langle z,y\rangle\hbox{ in } L^2(Q_1).
$$

By Proposition \ref{bvp-v2} above, applied with $v=\langle z,x\rangle$, $A=Q_1$ and
$\eta=\e/\rho$, for all $\delta>0$ there exists a sequence $v^\delta_\e$ such that $v^\delta_\e(y)=\langle z,y\rangle$ on $Q_1\setminus Q_{1-\delta}$ and
\begin{eqnarray*}
&&{1\over\e^d\rho^d}\int_{Q_\rho}\int_{Q_\rho}
b^\omega\Bigl({x\over\e},{y\over\e}\Bigr)\Bigl({u_\e(x)-u_\e(y)\over \e}\Bigr)^2 \dx\,dy
\\
&\ge& {\rho^d\over\e^d}
\int_{Q_1}\int_{Q_1}
b^\omega\Bigl({x_0\over\e}+{x\over\e/\rho},{x_0\over\e}+{y\over\e/\rho}\Bigr)\Bigl({v^\delta_\e(x)-v^\delta_\e(y)\over \e/\rho}\Bigr)^2\dx\,dy + o(1)
\end{eqnarray*}
as $\delta\to 0$ uniformly in $\e$.

If we set $R=R_\e=\rho/\e$ and change variables, we get
\begin{eqnarray*}
{1\over \rho^d}\mu_\e(x_0+Q_\rho)\ge {1\over R^d}
\int_{{x_0\over \e}+Q_{\rho\over\e}}\int_{{x_0\over \e}+Q_{\rho\over\e}}
b^\omega(x,y)(v_R(x)-v_R(y))^2\dx\,dy +o(1)
\end{eqnarray*}
as $\delta\to0$, where
$$
v_R(x)= v^\delta_\e\Bigl({x\over R}-{x_0\over\rho}\Bigr).
$$
For every fixed $K>0$ we have that
$$
v_R(x)=\langle z, x\rangle\hbox{ if dist}\Bigl(x,\partial \Bigl({x_0\over \e}+Q_{\rho\over\e}\Bigr)\Bigr)<K
$$
for $\e$ small enough (and hence $R$ large enough). Hence, we may use $v_R$ as a test function in the definition on $\widetilde{\cal M}^\omega_K(z,Q_R)$. We also note that suitably choosing $f$ in \eqref{slogro} we have that $x_R=x_0/\rho$ satisfies $|x_R|\le Rf_K(R)$ in Lemma \ref{lemma-one}, so that we finally obtain
\begin{eqnarray*}
\lim_{\e\to0}{1\over \rho^d}\mu_\e(x_0+Q_\rho)&\ge&\lim_{R\to+\infty}{{\cal M}^\omega_K(z,x_R+Q_R)\over R^d}+o(1)
= \gamma_K(z)+o(1)
\end{eqnarray*}
as $\delta\to0$.
Hence we have
$$
\Gamma\hbox{-}\liminf_{\e\to0}F_\e(u)\ge\int_U\gamma_K(\nabla u)\dx +o(1)
$$
By taking the supremum in $K$, using the Monotone Convergence Theorem, and by the arbitrariness of $\delta$ we get the desired lower bound.

\bigskip
The proof of the upper bound is obtained by a standard density argument by piecewise-affine functions (see also \cite{2018BP}) once it is shown for $D$ a $d$-dimensional simplex $S$ and $u(x)=\langle z,x\rangle$ a linear function.
%It is not restictive to suppose that $U$ is a cube. Since the argument is invariant by a translation, for simplicity of notation, we may suppose that $U$ is centred in $0$,
We consider $L$ large enough so that $Q_L\supset D$ for some $L>0$. We fix $m\in\NN$ and subdivide $Q_L$ into $m^d$ cubes $Q^m_i=x^m_i+Q_{L/m}$ of side-length $L/m$ and disjoint interiors. With fixed $K\in\NN$ we choose $u^i_\e\in L^2({1\over\e}Q^m_i)$ such that $v(x)=\langle z,x\rangle$ if  dist$(x,{1\over\e}\partial Q^m_i)<K$ and
\begin{eqnarray}\nonumber
\int_{{1\over\e}Q^m_i\times {1\over\e}{Q^m_i}}b^\omega(x,y)(u^i_\e(x)-u^i_\e(y))^2dx\,dy&\le& {\cal M}^\omega_K\Bigl(z,{1\over\e}x^m_i+Q_{L\over m\e}\Bigr)+1\\
&\le& {L^d\over m^d\e^d}(\gamma_K(z)+o(1))+1
\end{eqnarray}
as $\e\to 0$ and $K\to+\infty$.

We then define $u^m_\e\in L^2(Q)$ by setting
$$
u^m_\e(x)= \e \,u^i_\e\Bigl({x\over\e}\Bigr)\hbox{ if } x\in Q^m_i\,.
$$
We set $$I^m=\{I: Q^m_i\cap D\neq\emptyset\},$$
and compute
\begin{eqnarray*}
F^\omega_\e(u^m_\e)&\le &\sum_{i\in I^m} {1\over\e^{d+2}}
\int_{Q^m_i\times Q^m_i }b^\omega\Bigl({x\over\e},{y\over\e}\Bigr)(u^m_\e(x)-u^m_\e(y))^2dx\,dy\\
&&+{1\over\e^{d+2}}\sum_{i\neq j}\int_{\{x\in Q^m_i: {\rm dist}(x,\partial Q^m_i)<\e K\}}\int_{\{y\in Q^m_i: {\rm dist}(x,\partial Q^m_i)<\e K\}}
b^\omega\Bigl({x\over\e},{y\over\e}\Bigr)|z|^2|x-y|^2\dx\,dy\\
&&+{1\over\e^{d+2}}\int_{\{x,y\in Q_L: |x-y|>\e K\}}b^\omega\Bigl({x\over\e},{y\over\e}\Bigr)(u^m_\e(x)-u^m_\e(y))^2dx\,dy
\\
&\le &\sum_{i\in I^m} \e^d
\int_{{1\over\e}Q^m_i\times {1\over\e}{Q^m_i}}b^\omega(x,y)(u^i_\e(x)-u^i_\e(y))^2dx\,dy+CKm\e|z|^2+CK^{-\eta}
\\
&\le & \Bigl(|U|+O\Bigl({1\over m}\Bigr)\Bigr)\gamma_K(z)+o(1)+CKm\e|z|^2+CK^{-\eta}.
\end{eqnarray*}
Note that we have used assumption \eqref{ass-a} to estimate the second term in the sum, and Proposition \ref{prepro1}
with $U={L\over\e}Q$ and the coerciveness of $b^\omega$ to estimate the third term in the sum.

We may now choose $m=m_\e\to+\infty$ such that
$$
\limsup_{\e\to 0}F^\omega_\e(u^m_\e)\le L^d\gamma_K(z)+o(1)
$$
as $K\to+\infty$.
Note that, since $u^m_\e(x)=\langle z,x\rangle$ if dist$(x,\bigcup_i\partial(Q^m_i))<\e K$ then $u^m_\e\to \langle z,x\rangle$
in $L^2(D)$ and we obtain an upper bound with $\gamma_K(z)+o(1)$. Letting $K\to+\infty$ we finally have the desired estimate.
\end{proof}

\section{Random perforated domains}\label{hom_perf}
In this section we note that Theorem \ref{hom} can be applies to the homogenization on randomly perforated domains.

First we define random sets in $\mathbb R^d$.
Let $(\Omega,\mathcal{F},\mathbf{P})$  be a standard probability set, and assume that  $\tau_x$, $x\in\mathbb R^d$
is a {\em measure-preserving dynamical system} on this probability space; that is, $\{\tau_x\}_{x\in\mathbb R^d}$ is a group of measurable mappings $\tau_x:\Omega\mapsto\Omega$ such that
\begin{itemize}
  \item $\tau_x\circ\tau_y=\tau_{x+y}, \quad \tau_0=\mathrm{Id}$,
  \item $\mathbf{P}(\tau_x A)=\mathbf{P}(A)$  for all $x\in\mathbb R^d$ and $A\in \mathcal{F}$,
  \item $\tau_\cdot:\mathbb R^d\times\Omega\mapsto\Omega$ is a measurable map. We assume here that
  $\mathbb R^d\times\Omega$ is equipped with a product $\sigma$-algebra $\mathcal{B}\times\mathcal{F}$,
  where $\mathcal{B}$ is a Borel $\sigma$-algebra in $\mathbb R^d$.
\end{itemize}
We also assume that $\{\tau_x\}$ is {\em ergodic}; that is, the measure of any set $A\in\mathcal{F}$ which is invariant with respect to $\tau_x$ for all $x\in\mathbb R^d$ is equal to $0$ or $1$.

\begin{definition}[random sets and random perforations]
We say that  $E^\omega=\{x\in \mathbb R^d\,:\, \chi_{\Omega_1}(\tau_x\omega)=1\}$ is a {\em random set} in $\mathbb R^d$
if $\Omega_1\in\mathcal{F}$ is such that $\mathbf{P}(\Omega_1)\mathbf{P}(\Omega\setminus\Omega_1)>0$..
A random set  $E^\omega$ is called a {\em random perforated domain} if it possesses the following properties:
\begin{enumerate}
  \item Almost surely $\mathbb R^d\setminus E^\omega$ is a union of bounded open sets in $\mathbb R^d$;
  \item The diameters of these sets are uniformly bounded.
  \item The distance between any two distinct sets is bounded from below by a positive constant.
  \item The boundary of these sets are uniformly Lipschitz continuous; i.e., there exist constants  $L>0$ and $\rho_1, \,\rho_2>0$    such that for any point $x\in\partial E^\omega$  there exists a set $C$ which,
       up to translation by $x$ and rotation, is of the form $(-\rho_1,\rho_1)^{d-1}\times(-\rho_2,\rho_2)$ such that
       $C\cap E^\omega$ is the sub-graph of a $L$-Lipschitz function defined on $(-\rho_1,\rho_1)^{d-1}$.

\end{enumerate}
\end{definition}

We now assume that $E^\omega$ is a random perforated domain, and  we set
\begin{equation}\label{perfo_b}
b^\omega(x,y)=\chi\big._{E^\omega}(x) \chi\big._{E^\omega}(y)a(x-y).
\end{equation}
The key observation is that such $b^\omega$ is coercive.
This is implied by the following theorem in \cite{2018BP}.
%This will be obtained by the following extension theorem.

\begin{theorem}[extension theorem]\label{t_ext}
Let $E^\omega$ be a random perforated domain that satisfies condition {\rm (1)--(4)} above.
Let $b^\omega$ be defined by \eqref{perfo_b}.
Then there exists $k>0$ and $r>0$ such that almost surely %for each $\Domain'\subset\subset \Domain$ and
for all $u\in L^2(\Domain\cap \e E^\omega)$ there exists $v\in L^2(\Domain)$ such that
\begin{equation}\label{odin}
v=u\ \ \hbox{\rm on } \Domain\cap\e E^\omega,
\end{equation}
\begin{equation}\label{dva}
\int_{\Domain(k\e)}\int_{\{|\xi|\le r\}}
\Bigl({v(x+\e\xi)-v(x)\over\e}\Bigr)^2d\xi \dx\le C F^\omega_\e(u)
\end{equation}
and
\begin{equation}\label{tri}
\int_{\Domain(k\e)}|v|^2\dx\le C\int_{\Domain\cap\e E}|u|^2\dx.
\end{equation}
\end{theorem}

{
Theorem \ref{homthm} can be rephrased as follows.

\begin{theorem}\label{homthm}
Let $\Domain$ be an open set with Lipschitz boundary, let $E^\omega$ be a random perforated domain as above, and let $F^\omega_\e$ be given by
\begin{eqnarray}\label{def-Fe-random-set}
F^\omega_\e(u)={1\over \e^{d+2}}\int_{(D\cap\e E^\omega)\times (D\cap\e E^\omega)}a\Bigl({x-y\over\e}\Bigr)
(u(y)-u(x))^2dy \dx,
\end{eqnarray}
Then $F^\omega_\e$ almost surely $\Gamma$-converge with respect to the $L^2$-convergence
to the functional \eqref{homfun} on $H^1(\Domain)$, where $A_{\rm hom}$ is a symmetric matrix which satisfies
\begin{eqnarray}\label{homformula-set} \nonumber
 \langle A_{\rm hom}z,z\rangle&=&\lim_{K\to+\infty}\lim_{R\to+\infty}{1\over R^d}
\inf\Bigl\{\int_{Q_R\cap E^\omega}\int_{E^\omega}a(x-y)(v(x)-v(y))^2dx\,dy:
\\
&&\qquad\qquad\qquad\qquad
v(x)=\langle z,x\rangle \hbox{ \rm if  dist}(x,\partial Q_R)<K\Bigr\}.
 \end{eqnarray}
\end{theorem}

}

\subsection*{Acknowledgments.}
The authors acknowledge the MIUR Excellence Department Project awarded to the Department of Mathematics, University of Rome Tor Vergata, CUP E83C18000100006.


\begin{thebibliography}{26}\frenchspacing




\bibitem{AlbBel}
	G. Alberti and G. Bellettini,
	A non-local anisotropic model for phase transitions: asymptotic behaviour of rescaled energies,
	\emph{European Journal of Applied Mathematics}, \textbf{9.3} (1998), 261-284.

\bibitem{AliCic}
{R. Alicandro and M. Cicalese}. {A general integral representation result for continuum limits of discrete energies with superlinear growth}. {\it SIAM J. Math. Anal.} {\bf 36} (2004), 1--37.



\bibitem{ACG} R. Alicandro, M. Cicalese and A. Gloria.
Integral representation results for energies defined on stochastic lattices and application to nonlinear elasticity.
{\em Arch. Ration. Mech. Anal.} {\bf 200} (2011), 881--943

\bibitem{BBL} X. Blanc, C. Le Bris, and P. L. Lions. The energy of some microscopic stochastic lattices.
{\em Arch. Ration. Mech. Anal.} {\bf184} (2007), 303--339

%\bibitem{LN98} {\rm A. Braides},
%{\it Approximation of Free-Discontinuity Problems}, Lecture
%No\-tes in Mathematics {\bf 1694}, Springer Verlag, Berlin, 1998.

\bibitem{BBM}
J. Bourgain, H. Brezis and P. Mironescu.
Another look at Sobolev spaces, in  {\em Optimal Control and Partial Differential Equations}
(J.L. Menaldi, E. Rofman, A. Sulem eds.). IOS Press, Amsterdam, 2001, pp. 439--455.

\bibitem{LN98} {\rm A. Braides},
{\it Approximation of Free-Discontinuity Problems}, Lecture
No\-tes in Mathematics {\bf 1694}, Springer Verlag, Berlin, 1998.


\bibitem{GCB} {\rm A. Braides},
{\it $\Gamma$-convergence for Beginners}, Oxford University Press,
Oxford, 2002.


\bibitem{Handbook} {\rm A. Braides}, A handbook of $\Gamma$-convergence.
In {\it Handbook of Differential Equations.
Stationary Partial Differential Equations, Volume $3$}
(M. Chipot and P. Quittner, eds.), Elsevier, 2006.

\bibitem{BraidesSeoul}
A. Braides.
Discrete-to-continuum variational methods for lattice systems.
{\it Proceedings of the International Congress of Mathematicians
August {\rm13--21, 2014}, Seoul, Korea} (S. Jang, Y. Kim, D. Lee, and I. Yie, eds.)
Kyung Moon Sa, Seoul, 2014, Vol.~IV, pp.~997--1015

\bibitem{BCR} A. Braides, M. Cicalese, and M.Ruf.
Continuum limit and stochastic homogenization of discrete ferromagnetic thin films,
{\em Anal. PDE} {\bf11} (2018), 499--553

\bibitem{BDF}
 A. Braides and A. Defranceschi,
{\it Homogenization of Multiple Integrals}.
Oxford University Press,  Oxford, 1998.

\bibitem{BF}A. Braides and G. Francfort,
{Bounds on the effective behavior of a square conducting lattice}. 
{\it R. Soc. Lond. Proc. Ser. A
 Math. Phys. Eng. Sci.} {\bf 460} (2004), 1755--1769 

\bibitem{BK} A. Braides and L. Kreutz.
An integral-representation result for continuum limits of discrete energies with multi-body interactions.
 {\em SIAM J. Math. Analysis} {\bf 50} (2018), 1485--1520.

\bibitem{BMS} A. Braides, M. Maslennikov, and L. Sigalotti.
Homogenization by blow-up.
{\it Applicable Anal.} {\bf 87} (2008), 1341--1356.

\bibitem{2018BP}  A. Braides and  A. Piatnitski.
Homogenization of quadratic convolution energies in periodically perforated domains.
Preprint, 2019.

\bibitem{FKK2012} D. Finkelshtein, Yu. Kondratiev and O. Kutoviy.
Semigroup approach to birth-and-death stochastic dynamics in continuum. {\em J. Funct Anal.} {\bf 262} (2012), 1274--1308.

\bibitem{FM} I. Fonseca and S. M\"uller.
Quasi-convex integrands and lower semicontinuity in $L^1$.
{\em SIAM J. Math. Anal.} {\bf  23} (1992), 1081--1098.

\bibitem{PiatnitskiRemy} A. Piatnitski and E. Remy.
 Homogenization of elliptic difference operators.
 {\em SIAM J. Math. Analysis} {\bf 33} (2001), 53--83.

 \bibitem{GS}
	N. Garc\'ia Trillos and  D. Slep\v cev,
	Continuum limit of total variation on point couds,
	{\em Arch. Ration. Mech. Anal.} {\bf  220}  (2016), 193--241.


\bibitem{Go} M. Gobbino.
Finite difference approximation of the Mumford-Shah functional.
{\em Comm. Pure Appl. Math.} {\bf 51} (1998), 197--228

\bibitem{Kozlov} S.M. Kozlov. Averaging of difference schemes. {\em Math. USSR Sbornik} {\bf  57} (1987), 351--36


\bibitem{KKP2008} Yu. Kondratiev, O. Kutoviy and S. Pirogov. Correlation functions and invariant measures
in continuous contact model. {\em Infin. Dimens. Anal. Quantum Probab. Relat. Top.}, {\bf 11} (2008),
231--258.

\bibitem{KrPy} U. Krengel, R. Pyke, Uniform pointwise ergodic theorems for classes of averaging
sets and multiparameter subadditive processes. {\it Stoch. Proc. Appl.}, {\bf 26}, 298--296, (1987).

\end{thebibliography}
\end{document}